\documentclass[a4paper,9pt]{article}
\usepackage{geometry}         
\usepackage{graphicx}
\usepackage{float,caption}
\usepackage{amsmath,amssymb,mathrsfs,amsthm,bbold,amsfonts}
\usepackage{url} 
\usepackage {geometry}
\usepackage {setspace}
\usepackage{amsmath,amssymb,mathrsfs,amsthm,bbold,amsfonts}
\usepackage{pdfpages}
\usepackage[square,numbers,sectionbib]{natbib}
\usepackage{bibunits}
\usepackage{array}
\usepackage{subfig}
\usepackage[nottoc, notlof, notlot]{tocbibind}
\usepackage{amssymb}
\usepackage[babel=true]{csquotes} 
\usepackage{subfig}
\newtheorem{proposition}{Proposition}[section]

\newtheorem{theoreme}[proposition]{Theorem}
\newtheorem{theorem}[proposition]{Th\'eor\`eme}

\newtheorem{remarque}[proposition]{Remark}
\newtheorem{notation}[proposition]{Notation}
\begin{document}
\title{Lipschitz stability estimate in the inverse Robin problem for the Stokes system}
\author{Anne-Claire Egloffe}
\maketitle
\textbf{Abstract}
\par
We are interested in the inverse problem of recovering a Robin coefficient  defined on some non accessible part of the boundary from available data on another part of the boundary 
in the nonstationary Stokes system. We prove a Lipschitz stability estimate under the \textit{a priori} assumption that the Robin coefficient lives in some compact and convex subset
of a finite  dimensional vectorial subspace of the set of continuous functions. To do so, we use a theorem  proved by L. Bourgeois  which establishes Lipschitz stability 
estimates for a class of inverse problems in an abstract framework.

\vspace*{5mm}

\par
\textbf{R\'esum\'e} \textit{Estimation de stabilit\'e Lipschitzienne pour le syst\`eme de Stokes avec des conditions aux limites de types Robin}
\par
Nous nous int\'eressons \`a l'identification d'un coefficient de Robin d\'efini sur une partie non accessible du bord \`a partir de mesures  disponibles sur une autre partie du bord dans le syst\`eme de Stokes non stationnaire. Nous prouvons une estimation de stabilit\'e Lipschitzienne sous l'hypoth\`ese \textit{a priori} que le coefficient de  Robin est d\'efini dans un sous-ensemble compact et convexe d'un sous-espace vectoriel de dimension finie de l'espace des fonctions continues. Pour ce faire, nous utilisons un th\'eor\`eme  prouv\'e par L. Bourgeois permettant d'\'etablir des in\'egalit\'es de stabilit\'e Lipschitzienne pour une classe de probl\`emes inverses dans un cadre abstrait. 

\section{Version fran\c caise abr\'eg\'ee}
%
\par
Soit $T>0$, $\Omega \subset \mathbb{R}^d$, avec $d \in \mathbb{N}^*$, un ouvert  born\'e Lipschitzien et  connexe tel que $\partial \Omega= \Gamma_l \cup \Gamma_{0} \cup \Gamma_{out}$
avec  $\displaystyle \Gamma_{out}=\bigcup_{i=1}^N \Gamma_i$  et $\nu$ est la normale ext\'erieure \`a  $\Omega$.
On consid\`ere le syst\`eme de Stokes suivant : 
\begin{equation} \label{eqStokes}
	\left \{   
		\begin{array}{ccll}
                  \partial_t u    - \Delta u + \nabla p							&	=	&	0,							& \textrm{ dans } (0,T)\times \Omega, \\ 
                                          div \textrm{ } u 							&	=	&	0,							& \textrm{ dans } (0,T)\times \Omega, \\  
                                          u										&	=	&	0,							& \textrm{ dans } (0,T)\times\Gamma_l, \\
                         \displaystyle       \partial_{\nu} u- p \nu			&	=	&	g, 							& \textrm{ sur  } (0,T)\times \Gamma_0, \\   
                         \displaystyle       \partial_{\nu} u- p \nu	+ qu 	&	=	&	0,							& \textrm{ sur  } (0,T)\times \Gamma_{out},  \\
                         						u(0,.)						&	=	&	u_0,							& \textrm{ dans } \Omega.
                   \end{array}     
         \right. 
 \end{equation}
%
%
							
%
%
%
Le probl\`eme inverse qui nous int\'eresse est le suivant : on cherche \`a identifier le coefficient de Robin $q$ d\'efini sur la partie non accessible du bord $ \Gamma_{out}$ 
\`a partir de mesures
disponibles sur $ \Gamma_0$ pour $(u,p)$ solution du syst\`eme \eqref{eqStokes}. 
%
%
Ce type de probl\`eme inverse appara\^it naturellement dans la mod\'elisation d'\'ecoulements biologiques, comme par exemple l'\'ecoulement sanguin dans le syst\`eme cardiovasculaire 
(voir~\cite{quarteroni} et~\cite{irene}) ou encore l'\'ecoulement de l'air dans les poumons (voir~\cite{Baffico}). Nous renvoyons \`a \cite{egloffe_these} pour une introduction \`a la mod\'elisation 
de l'\'ecoulement de l'air dans les poumons et aux diff\'erentes  conditions aux limites qui peuvent \^etre prescrites. 
%
La r\'esolution de ce type de probl\`eme inverse dans le cas stationnaire a d\'ej\`a  \'et\'e \'etudi\'ee dans~\cite{us}, ~\cite{us2} et~\cite{rr}. Dans les deux premiers travaux, une in\'egalit\'e de 
stabilit\'e logarithmique est obtenue alors qu'une in\'egalit\'e de stabilit\'e Lipschitzienne est \'etablie dans~\cite{rr} sous l'hypoth\`ese \textit{ a priori } que le coefficient de Robin est constant
par morceaux sur $\Gamma_{out}$. Dans chacun de ces papiers, les mesures intervenant dans les in\'egalit\'es de stabilit\'e sont la vitesse $u$, la pression $p$ et la d\'eriv\'ee normale de la pression $\displaystyle \frac{\partial p }{\partial n}$ sur $\Gamma \subseteq  \Gamma_0$. Le cas du syst\`eme de Stokes non stationnaire  a \'et\'e abord\'e dans~\cite{us} dans le cas particulier o\`u le 
coefficient de Robin ne d\'epend pas du temps. L'id\'ee, introduite dans~\cite{bellassoued_cheng_choulli} dans le cas de l'\'equation de Laplace, consiste \`a \'etendre l'in\'egalit\'e de stabilit\'e valable pour le probl\`eme stationnaire au probl\`eme non stationnaire en utilisant une  in\'egalit\'e provenant de la th\'eorie des semigroupes analytiques. Cela conduit \`a faire des mesures en temps infini.
 \par
L'originalit\'e de l'in\'egalit\'e de stabilit\'e Lipschitzienne pr\'esent\'ee dans cette Note est multiple : d'une part, nous obtenons une in\'egalit\'e de stabilit\'e valable pour le syst\`eme de 
Stokes non stationnaire en temps fini avec un coefficient de Robin d\'ependant du temps et d'autre part, l'unique mesure intervenant dans l'in\'egalit\'e de stablit\'e est la vitesse 
$u$ sur $(0,T) \times\Gamma$, avec $\Gamma \subseteq  \Gamma_0$. De plus, l'ensemble des coefficients de Robin pour lequel l'in\'egalit\'e de stabilit\'e Lipschitzienne est valide est un peu plus 
g\'en\'eral que dans~\cite{rr} : les coefficients de Robin ne sont plus n\'ecessairement constants par morceaux mais appartiennent \`a un sous-ensemble compact et convexe d'un sous-espace vectoriel de 
dimension finie de l'ensemble des fonctions continues. Enfin, nous avons besoin d'hypoth\`eses de r\'egularit\'e moins fortes sur le bord du domaine $\Omega$ et sur le flux $g$. 
\par
Afin d'\^etre plus pr\'ecis, nous introduisons quelques notations.
\begin{notation}
On note 
\begin{equation*}
L^{\infty}_+((0,T)\times \Gamma_{out})=\{q \in L^{\infty}((0,T)\times \Gamma_{out}) \textrm{; } \exists m>0 \textrm{, } q \geq m \textrm{ p. p. sur }(0,T) \times\Gamma_{out}  \},
\end{equation*}
et
\begin{equation*}
\mathcal{C}^0(0,T; \mathcal{C}^0_{pc}(\Gamma_{out}))=\{q :(0,T)\times \Gamma_{out} \to \mathbb{R} \textrm{; } q\vert_{(0,T)\times\Gamma_i} \in \mathcal{C}^0((0,T)\times \Gamma_{i})  \textrm{ pour } 1 \leq i \leq N \}.
\end{equation*}
\end{notation}
Le r\'esultat principal de cette Note est r\'esum\'e dans le th\'eor\`eme suivant :
\begin{theorem} \label{main_result_fr}
Soit $M \in \mathbb{N}^*$. On consid\`ere $V_M$ un sous-espace vectoriel de $\mathcal{C}^0(0,T; \mathcal{C}^0_{pc}(\Gamma_{out}))$ engendr\'e par $M$ fonctions lin\'eairement ind\'ependantes 
et $K_M$ un sous-espace convexe et compact de  $V_M \cap L^{\infty}_+((0,T)\times \Gamma_{out})$. Soit $\Gamma \subseteq \Gamma_0$ une partie ouverte non vide du bord du domaine,  $u_0 \in H^1(\Omega)$ 
tel que $div$ $u_0=0$ dans $\Omega$ et $g \in H^1(0,T;L^2(\Gamma_0))$  tel que  $g(t)$ est non identiquement z\'ero pour tout $t \in (0,T)$.  Soit $(u_k,p_k)$ la solution faible du syst\`eme~\eqref{eqStokes} avec $q=q_k \in K_M$ pour $k=1,2$. 
Alors, il existe une constante $C>0$ telle que
\begin{equation*}
\|q_1-q_2\|_{L^{\infty}((0,T) \times \Gamma_{out} )} \leq C \|u_1-u_2\|_{L^2((0,T)Ê\times \Gamma)}.
\end{equation*}
\end{theorem}
La preuve du Th\'eor\`eme \ref{main_result_fr} est bas\'ee sur un th\'eor\`eme abstrait prouv\'e par L. Bourgeois dans~\cite{boubou} que nous rappelons dans le Th\'eor\`eme \ref{abstract_th} et repose sur  le fait que l'application 
\begin{equation*}
\begin{array}{rcl}
T :  L^{\infty}_+((0,T)\times \Gamma_{out}) & \to &L^2((0,T)\times\Gamma) \\
  q & \to & u\vert_{\Gamma}
\end{array}
\end{equation*}
o\`u $(u,p)$ est solution du syst\`eme~\eqref{eqStokes} et avec $\Gamma \subseteq \Gamma_0$, est injective, de classe $\mathcal{C}^1$ et sa d\'eriv\'ee est \'egalement
injective. Notons que le r\'esultat \'enonc\'e dans \cite{boubou} permet d'\'etablir des in\'egalit\'es de stabilit\'e Lipschitziennes pour une classe de probl\`emes inverses. Il permet notamment de retrouver les r\'esultats de stabilit\'e
d\'evelopp\'es
dans \cite{sincich} et \cite{alessandrini_et_al} sans avoir recours \`a des arguments de quantification de r\'esultats de continuation unique. L'auteur pr\'ecise que l'on peut trouver l'id\'ee originale d\'evelopp\'ee dans \cite{conca} dans le cas particulier de la d\'etection d'un obstacle se d\'epla\c cant dans un fluide \`a partir de mesures disponibles sur le bord du domaine. De plus, des th\'eor\`emes abstraits du m\^eme  type mais avec des hypoth\`eses diff\'erentes peuvent \^etre trouv\'es dans \cite{stefanov}.
%
%
%
\section{Introduction}
Let $T>0$, $\Omega \subset \mathbb{R}^d$, with $d \in \mathbb{N}^*$, be a Lipschitz   bounded connected   open  set such that $\partial \Omega= \Gamma_l \cup \Gamma_{0} \cup \Gamma_{out}$ and 
$\displaystyle \Gamma_{out}= \bigcup_{i=1}^N \Gamma_i$. We are interested in the inverse problem of identifying the  
Robin coefficient $q$ defined on some non accessible part of the boundary $\Gamma_{out}$ from   available data on $ \Gamma_0$ for $(u,p)$ solution of the Stokes system~\eqref{eqStokes}.   
%
%
Such kinds of systems naturally appear in the modeling of biological problems like, for example, blood flow in the cardiovascular system (see~\cite{quarteroni} and~\cite{irene}) or airflow in the lungs  (see~\cite{Baffico}). For an introduction on the modeling of the airflow in the lungs and on different boundary conditions which may be prescribed, we refer to \cite{egloffe_these}. 
%
%
Similar inverse  problems  have already been studied in the stationary case in~\cite{us}, ~\cite{us2} and~\cite{rr}. In~\cite{us} and~\cite{us2}, a logarithmic stability estimate is obtained, whearas a Lipschitz stability estimate is established in~\cite{rr} under the \textit{ a priori } assumption that the  Robin coefficient is piecewise constant on $\Gamma_{out}$. In each cases, the measurements involved in the stability estimates are the velocity $u$, the pressure $p$ and the normal derivative of the pressure $\displaystyle \frac{\partial p }{\partial n}$ on $\Gamma \subseteq \Gamma_0$. The case of the nonstationary Stokes system has been addressed in~\cite{us} in the particular case where the Robin coefficient does not depend  time. The idea, introduced in~\cite{bellassoued_cheng_choulli} in the case of the Laplace equation,  is to extend the stability estimate valid for the stationary problem to the nonstationary problem by using an inequality from the theory of analytic semigroups. This leads to infinite time measurements.
\par
The originality of the Lipschitz stability estimate presented in this Note is multiple: on the one hand, we obtain a stability estimate valid for the nonstationary Stokes system in finite time with a time-dependent Robin coefficient  and secondly, the only measurement  involved in the stability estimate is the velocity $u$ on $(0,T)\times\Gamma $, with $\Gamma \subseteq \Gamma_0$. In addition, the set of admissible Robin coefficients is more general than in~\cite{rr}: Robin coefficients are not necessarily piecewise constant but belong to some compact and convex subset
of a finite  dimensional vectorial subspace of the set of continuous functions. Finally, we relax the  regularity assumptions needed both on the boundary of the domain $\Omega$ and on the flux $g$. 
\par
To be more precise, we introduce some notations.
\begin{notation}
We denote by
\begin{equation*}
L^{\infty}_+((0,T)\times \Gamma_{out})=\{q \in L^{\infty}((0,T)\times \Gamma_{out})\textrm{; }\exists m>0 \textrm{, } q \geq m \textrm{ a. e. on }(0,T) \times\Gamma_{out}  \},
\end{equation*}
and
\begin{equation*}
\mathcal{C}^0(0,T; \mathcal{C}^0_{pc}(\Gamma_{out}))=\{q:(0,T)\times \Gamma_{out} \to \mathbb{R}\textrm{; } q\vert_{(0,T)\times\Gamma_i} \in \mathcal{C}^0((0,T)\times \Gamma_{i})  \textrm{ for } 1 \leq i \leq N \}.
\end{equation*}
\end{notation}
\par
The main result of this Note is summarized in the following theorem.
\begin{theoreme} \label{main_result}
Let  $M \in \mathbb{N}^*$.  Let $V_M$ be a subspace of $\mathcal{C}^0(0,T; \mathcal{C}^0_{pc}(\Gamma_{out}))$ spanned by some $M$ linearly independent functions and $K_M$ be any compact and convex subset of $V_M \cap L^{\infty}_+((0,T)\times \Gamma_{out})$.
Let $\Gamma \subseteq \Gamma_0$ be a nonempty open subset of the boundary, $u_0 \in H^1(\Omega)$ be such that $div$ $u_0=0$ in $\Omega$ and $g \in H^1(0,T;L^2(\Gamma_0))$ be such that $g(t)$ is not identically zero for all $t \in (0,T)$.  Let  $(u_k,p_k)$  be the weak solutions of system~\eqref{eqStokes} with $q=q_k \in K_M$ for $k=1,2$. Then, there exists a constant $C>0$ such that 
\begin{equation*}
\|q_1-q_2\|_{L^{\infty}((0,T) \times \Gamma_{out} )} \leq C \|u_1-u_2\|_{L^2((0,T)Ê\times \Gamma)}
\end{equation*}
\end{theoreme}
The proof of Theorem \ref{main_result} is based on an abstract theorem proved by L. Bourgeois in~\cite{boubou} that we recall in Theorem \ref{abstract_th} and rely on the fact that the application which,
to a Robin coefficient maps the velocity on $ (0, T) \times \Gamma $ with $ \Gamma \subseteq \Gamma_0 $ is injective, of class $ \mathcal{C}^1 $ and its derivative is also injective. The result stated in \cite{boubou} establishs Lipschitz stability estimates for a class of inverse problems. For instance, it allows to find again the stability results developed in \cite{sincich} and \cite{alessandrini_et_al} without resorting to quantification of unique continuation results. The author points out that one can find the original idea developed in \cite{conca}  in the particular case of the detection of a moving obstacle in a fluid from measurements available on the boundary of the domain. Moreover,  abstract theorems of the same type but with different assumptions can be found in \cite{stefanov}.
The sequel of this paper is organized as follows. We present in Section~\ref{pre} some preliminary results which will be useful to prove Theorem~\ref{main_result}. Then the proof of Theorem~\ref{main_result} is given in Section~\ref{lip}.
\section{Preliminary results}
\label{pre}
In the section, we state results which will be useful in the proof of Theorem~\ref{main_result}. We begin by stating regularity result for a slightly more general Stokes system than system~\eqref{eqStokes}   (we add non homogeneous Robin boundary condition on $(0,T) \times \Gamma_{out}$):
\begin{equation} \label{eqStokes2}
	\left \{   
		\begin{array}{ccll}
                  \partial_t u    - \Delta u + \nabla p							&	=	&	0,							& \textrm{ in } (0,T)\times \Omega, \\ 
                                          div \textrm{ } u 							&	=	&	0,							& \textrm{ in } (0,T)\times \Omega, \\  
                                          u										&	=	&	0,							& \textrm{ in } (0,T)\times\Gamma_l, \\
                         \displaystyle       \partial_{\nu} u- p \nu			&	=	&	g, 							& \textrm{ on  } (0,T)\times \Gamma_0, \\   
                         \displaystyle       \partial_{\nu} u- p \nu	+ qu 	&	=	&	\kappa,							& \textrm{ on  } (0,T)\times \Gamma_{out},  \\
                         						u(0,.)						&	=	&	u_0,							& \textrm{ in } \Omega.
                   \end{array}     
         \right. 
 \end{equation}

\begin{proposition} \label{reg}
Assume that $g \in H^1(0,T; L^2(\Gamma_0))$, $\kappa \in H^1(0,T ; L^2(\Gamma_{out}))$, $u_0 \in H^1(\Omega)$ be such that $div$ $u_0=0$ in $\Omega$ and $q \in  L^{\infty}_+((0,T)\times \Gamma_{out})$. Then,  system~\eqref{eqStokes2} has a unique solution which belongs to $L^2(0,T; H^1(\Omega))  \cap H^1(0,T; L^2(\Omega)) \times L^2(0,T; L^2(\Omega))$. Moreover, there exists $C>0$, independent of $q$,  such that the following inequality holds
\begin{equation*}
 \|u\|_{L^2(0,T; H^1(\Omega))} \leq C( \|u_0\|_{L^2(\Omega)} + \|g\|_{L^{2}(0,T;L^2(\Gamma_0))}+\|\kappa\|_{L^{2}(0,T;L^2(\Gamma_{out}))}).
\end{equation*}
\end{proposition}
\begin{proof}[Proof of Proposition~\ref{reg}]
The proof is mainly contained in the appendix of~\cite{us}. The main difference here is that we work with non homogeneous Robin boundary conditions  which leads to slight modifications.
\end{proof}
\begin{remarque}
Note that due to the mixed boundary conditions, the fact that $\partial_t u \in \L^2(0,T;L^2(\Omega))$ does not imply that $(u,p) \in L^2(0,T; H^2(\Omega))   \times L^2(0,T; H^1(\Omega))$
\end{remarque}
The following Proposition~\ref{unicite} concerns the identifiability of the inverse problem we are interested in.
\begin{proposition} \label{unicite}
Let $\Gamma \subseteq \Gamma_0$ be a nonempty open subset of the boundary, $u_0 \in H^1(\Omega)$ be such that $div$ $u_0=0$ in $\Omega$. Assume that  $g \in H^1(0,T;L^2(\Gamma_0))$ is such that $g(t)$ is not identically zero for all $t \in (0,T)$.  Let  $(u_k,p_k)$  be the weak solutions of system~\eqref{eqStokes} with $q=q_k \in \mathcal{C}^0(0,T;\mathcal{C}^0_{pc}(\Gamma_{out}))$ for $k=1,2$.
We assume that  $u_1=u_2$ on $(0,T) \times \Gamma$. Then $q_1=q_2$ on $ (0,T)\times \Gamma_{out}$.
\end{proposition}
\begin{proof}[Proof of Proposition~\ref{unicite}]
The proof is based on the unique continuation result for the Stokes system proved  by C.~Fabre and G. Lebeau in~\cite{fabre}. Thanks to the previous proposition, $(u,p) \in L^2(0,T; H^1(\Omega))  \cap H^1(0,T; L^2(\Omega)) \times L^2(0,T; L^2(\Omega))$, which is enough regularity to prove a similar result to Corollary 3.2 in~\cite{us}. Then, we proceed exactly as in the proof of Proposition 3.3 in~\cite{us} where the proof is done in the particular case when the Robin coefficient does not depend on time, by arguing by contradiction.
\end{proof}
As announced previously, the proof of our Lipschitz stability estimate is based on a theorem proved by L. Bourgeois in~\cite{boubou} which establishs Lipschitz stability estimate in an abstract framework for parameters defined on some finite dimensional subspace of the set of the continuous functions. For the sake of completeness, we state this theorem below.
\begin{theoreme} \label{abstract_th}
Let $(V,\|$ $\|_V )$ and $(H,\|$ $\|_H)$ be two Banach spaces. Let $U$ be an open subset of $V$ and $V_M$ a finite dimensional subspace of $V$ of dimension $M$. Let $K_M$ be a compact and convex subset of $V_M \cap U$. We consider a mapping $T: U \to H$ which satisfies the following assumptions:
\begin{enumerate}
\item $T: V_M \cap U \to H$ is injective,
\item $T: U \to H$  is $\mathcal{C}^1$: $T$ is differentiable in the sense of Fr\'echet at any point $x \in U$, the Fr\'echet derivative being denoted $dT_x: V \to H $ and the mapping $x \in U \to dT_x \in \mathcal{L}(V,H)$ is continuous.
\item For all $x \in V_M \cap U$, the operator $dT_x: V_M \to H$ is injective.
\end{enumerate} 
Then, there exists $C>0$ such that 
$\forall x, y \in K_M, \|x-y\|_V \leq C \|T(x)-T(y)\|_H.
$
\end{theoreme}
%
%
\section{Proof of the main result }
\label{lip}
In this section, we establishes the proof of Theorem~\ref{main_result}.
The proof consists of applying the abstract Theorem~\ref{abstract_th} with  $V=L^{\infty}((0,T) \times \Gamma_{out})$, $H=L^2((0,T) \times\Gamma)$, $U=L^{\infty}_+((0,T)\times \Gamma_{out})$. 
We consider the operator
\begin{equation*}
\begin{array}{rcl}
T:  L^{\infty}_+((0,T)\times \Gamma_{out}) & \to &L^2((0,T)\times\Gamma) \\
  q & \to & u\vert_{\Gamma}
\end{array}
\end{equation*}
where $(u,p)$ is solution of system~\eqref{eqStokes}. We are going to prove that:
\begin{enumerate}
\item $T \vert_{ V_M \cap U}$ is injective,
\item $T$ is differentiable  at any point $q \in L^{\infty}_+((0,T)\times \Gamma_{out})$ and its Fr\'echet derivative is the operator 
\begin{equation*}
\begin{array}{rcl}
dT_q:  L^{\infty}((0,T)\times \Gamma_{out}) & \to &L^2((0,T)\times\Gamma) \\
  h & \to & v_h\vert_{\Gamma},
\end{array}
\end{equation*}
where $(v_h,\tau_h)$ is solution to 
\begin{equation} \label{def_diff}
	\left \{   
		\begin{array}{ccll}
                  \partial_t v_h    - \Delta v_h + \nabla \tau_h							&	=	&	0,							& \textrm{ in } (0,T)\times \Omega, \\ 
                                          div \textrm{ } v_h 									&	=	&	0,							& \textrm{ in } (0,T)\times \Omega, \\  
                                          v_h												&	=	&	0,							& \textrm{ in } (0,T)\times\Gamma_l, \\
                         \displaystyle      \partial_{\nu} v_h - \tau_h\nu				&	=	&	0, 							& \textrm{ on  } (0,T)\times \Gamma_0, \\   
                         \displaystyle      \partial_{\nu} v_h - \tau_h\nu	+ qv_h 		&	=	&	-hu,							& \textrm{ on  } (0,T)\times \Gamma_{out},  \\
                         						v_h(0,.)								&	=	&	0,							& \textrm{ in } \Omega,
                   \end{array}     
         \right. 
 \end{equation}
where $u$ is solution to system~\eqref{eqStokes}. Moreover, the mapping 
\begin{equation} \label{mapping_diff}
\begin{array}{rcl}
dT: L^{\infty}_+((0,T)\times \Gamma_{out})  & \to &\mathcal{L}( L^{\infty}_+((0,T)\times \Gamma_{out}) ,L^2((0,T)\times\Gamma) )\\
  q & \to & dT_q,
\end{array}
\end{equation}
 is continuous.
\item For all $q \in V_M \cap L^{\infty}_+((0,T)\times \Gamma_{out})$, the operator $dT_q: V_M \to L^2((0,T)\times\Gamma)$ is injective.
\end{enumerate} 
Step 1 is a direct consequence of Proposition~\ref{unicite}.
Let us prove step 2. Let $q,h \in L^{\infty}_+((0,T)\times \Gamma_{out})$ and $(u,p)$ (resp. $(u_h,p_h)$) be the weak solution of system~\eqref{eqStokes} associated to $q$ (resp. to $q=q+h$). We denote by $(w_h,\pi_h)=(u_h-u,p_h-p)$ which is solution of the following Stokes system: 
\begin{equation*}
	\left \{   
		\begin{array}{ccll}
                  \partial_t w_h    - \Delta w_h + \nabla \pi_h							&	=	&	0,							& \textrm{ in } (0,T)\times \Omega, \\ 
                                          div \textrm{ } w_h 							&	=	&	0,							& \textrm{ in } (0,T)\times \Omega, \\  
                                          w_h										&	=	&	0,							& \textrm{ in } (0,T)\times\Gamma_l, \\
                         \displaystyle      \partial_{\nu} w_h - \pi_h{\nu}			&	=	&	0, 							& \textrm{ on  } (0,T)\times \Gamma_0, \\   
                         \displaystyle      \partial_{\nu} w_h - \pi_h{\nu}	+ qw_h 	&	=	&	-hu_h,							& \textrm{ on  } (0,T)\times \Gamma_{out},  \\
                         						w_h(0,.)						&	=	&	0,							& \textrm{ in } \Omega.
                   \end{array}     
         \right. 
 \end{equation*}
Let $M_1>0$ be such that $\|u_0\|_{L^2(\Omega)} + \|g\|_{L^2(0,T; L^2(\Gamma_{0}))} \leq M_1$.
Then, thanks to Proposition~\ref{reg}, there exists $C>0$ such that
\begin{equation}Ê\label{what}
\|u_h-u\|_{L^2(0,T;H^1(\Omega))} \leq C \|h\|_{L^{\infty}((0,T)\times \Gamma_{out})} \|u_h\|_{L^2(0,T;H^1(\Omega))} \leq C(M_1)\|h\|_{L^{\infty}((0,T)\times \Gamma_{out})}. 
\end{equation}
Now, let us consider $(e_h,\rho_h)=(u_h-u-v_h,p_h-p-\tau_h)$, where $(v_h,\tau_h)$ is solution to system~\eqref{def_diff}. First, we readily check that the operator $h \in L^{\infty}((0,T)\times \Gamma_{out})  \to v_h\vert_{\Gamma} \in L^2((0,T)\times\Gamma)$ is  linear continuous. Secondly, $(e_h,\rho_h)$ solves the problem
\begin{equation*}
	\left \{   
		\begin{array}{ccll}
                  \partial_t e_h    - \Delta e_h + \nabla \rho_h							&	=	&	0,							& \textrm{ in } (0,T)\times \Omega, \\ 
                                          div \textrm{ } e_h 							&	=	&	0,							& \textrm{ in } (0,T)\times \Omega, \\  
                                          e_h										&	=	&	0,							& \textrm{ in } (0,T)\times\Gamma_l, \\
                         \displaystyle      \partial_{\nu} e_h - \rho_h\nu			&	=	&	0, 							& \textrm{ on  } (0,T)\times \Gamma_0, \\   
                         \displaystyle      \partial_{\nu} e_h - \rho_h\nu	+ qe_h 	&	=	&	-h(u_h-u),							& \textrm{ on  } (0,T)\times \Gamma_{out},  \\
                         						e_h(0,.)						&	=	&	0,							& \textrm{ in } \Omega,
                   \end{array}     
         \right. 
 \end{equation*}
which implies,  thanks to Proposition~\ref{reg} and inequality~\eqref{what},  that
\begin{equation*}
\|e_h\|_{L^2((0,T);H^1(\Omega))}  \leq C \|h\|_{L^{\infty}((0,T)\times \Gamma_{out})} \|u_h-u\|_{L^2(0,T;H^1(\Omega))} \leq C(M_1)\|h\|_{L^{\infty}((0,T)\times \Gamma_{out})}^2 , 
\end{equation*}
which proves that $T$ is Fr\'echet differentiable and $dT_q(h)=v_h \vert_{\Gamma}$. 
\par
Let us prove now the continuity of the mapping $dT$ defined in~\eqref{mapping_diff}. Let $(v_h,\tau_h)$ (resp. $(v_h^l,\tau_h^l)$) be the solution of system~\eqref{def_diff} associated to $q$ (resp. $q=q+l$) and where $(u,p)$ (resp. $(u,p)=(u_l,p_l)$) is the solution to system~\eqref{eqStokes} associated to $q$ (resp. $q=q+l$) . 
We have that $(v_h^l-v_h,p_h^l-p_h)$ is the solution to the following Stokes system:
\begin{equation*}
	\left \{   
		\begin{array}{ccll}
                  \partial_t (v_h^l-v_h)   - \Delta (v_h^l-v_h)+ \nabla (p_h^l-p_h)							&	=	&	0,							& \textrm{ in } (0,T)\times \Omega, \\ 
                                          div \textrm{ } (v_h^l-v_h)							&	=	&	0,							& \textrm{ in } (0,T)\times \Omega, \\  
                                          v_h^l-v_h										&	=	&	0,							& \textrm{ in } (0,T)\times\Gamma_l, \\
                         \displaystyle      \partial_{\nu} (v_h^l-v_h)- (p_h^l-p_h)\nu			&	=	&	0, 							& \textrm{ on  } (0,T)\times \Gamma_0, \\   
                         \displaystyle      \partial_{\nu} (v_h^l-v_h)- (p_h^l-p_h)\nu	+ q(v_h^l-v_h)	&	=	&	-lv_h^l-h(u_l-u),							& \textrm{ on  } (0,T)\times \Gamma_{out},  \\
                         						 (v_h^l-v_h)(0,.)						&	=	&	0,							& \textrm{ in } \Omega.
                   \end{array}     
         \right. 
 \end{equation*}
This implies, thanks to Proposition~\ref{reg},  
\begin{equation*}
\|v_h^l-v_h\|_{L^2(0,T;H^1(\Omega))} \leq C \|l\|_{L^{\infty}((0,T)\times \Gamma_{out})} \|v_h^l\|_{L^2(0,T;H^1(\Omega))} \\
+ C \|h\|_{L^{\infty}((0,T)\times \Gamma_{out})} \|u_l-u\|_{L^2(0,T;H^1(\Omega))}, 
\end{equation*}
which leads to, applying again Proposition~\ref{reg} and inequality~\eqref{what} with $h=l$:
\begin{equation*}
\|v_h^l-v_h\|_{L^2(0,T;H^1(\Omega))} \leq C(M_1) \|l\|_{L^{\infty}((0,T)\times \Gamma_{out})}\|h\|_{L^{\infty}((0,T)\times \Gamma_{out})}, 
\end{equation*}
where $C$ is uniform with respect to $h$ and $l$. Otherwise, we have proved that
\begin{equation*}
|||dT_{q+l}-dT_q|||  \leq C \|l\|_{L^{\infty}((0,T)\times \Gamma_{out})}, 
\end{equation*}
where $|||$ $|||$ denotes the operator norm. Thus the mapping $dT$ is continuous.
\par
It remains to prove Step 3. Let $q \in L^{\infty}_+((0,T)\times \Gamma_{out}) \cap V_M $. Assume that  $h \in V_M$ is such that $v_h\vert_{(0,T)\times \Gamma}=0$. Then, since $(v_h\vert_{(0,T)\times\Gamma}, (\partial_{\nu} v_h- \tau_h \nu)\vert_{(0,T)\times \Gamma})=(0,0)$, we obtain from unique continuation result that $v_h=0$ in $(0,T)\times\Omega$ and then $hu=0$ on $(0,T)\times\Gamma_{out}$. We conclude that $h=0$ by contradiction, exactly as for the injectivity of the mapping $T$ (see~\cite{us}).
\bibliography{biblio1.bib}

\begin{thebibliography}{10}

\bibitem{alessandrini_et_al}
G.~Alessandrini, E.~Beretta, E.~Rosset, and S.~Vessella.
\newblock Optimal stability for inverse elliptic boundary value problems with
  unknown boundaries.
\newblock {\em Ann. Scuola Norm. Sup. Pisa Cl. Sci. (4)}, 29(4):755--806, 2000.

\bibitem{Baffico}
L.~Baffico, C.~Grandmont, and B.~Maury.
\newblock Multiscale modeling of the respiratory tract.
\newblock {\em Math. Models Methods Appl. Sci.}, 20(1):59--93, 2010.

\bibitem{bellassoued_cheng_choulli}
M.~Bellassoued, J.~Cheng, and M.~Choulli.
\newblock Stability estimate for an inverse boundary coefficient problem in
  thermal imaging.
\newblock {\em J. Math. Anal. Appl.}, 343(1):328--336, 2008.

\bibitem{us2}
M.~Boulakia, A.-C. Egloffe, and C.~Grandmont.
\newblock {Unique continuation estimates for the Stokes system. Application to
  an inverse problem.}
\newblock {\em Preprint}, 2012.

\bibitem{us}
M.~Boulakia, A.-C. Egloffe, and C.~Grandmont.
\newblock {Stability estimates for a Robin coefficient in the two-dimensional
  Stokes system}.
\newblock {\em Mathematical control and related field}, 2(1), 2013.

\bibitem{boubou}
L.~Bourgeois.
\newblock A remark on lipschitz stability for inverse problems.
\newblock {\em Inria research report, RR-8104}, Accepted for publication in
  Comptes Rendus de l'Acad\'emie des sciences, 2013.

\bibitem{conca}
C.~Conca, P.~Cumsille, J.~Ortega, and L.~Rosier.
\newblock On the detection of a moving obstacle in an ideal fluid by a boundary
  measurement.
\newblock {\em Inverse Problems}, 24(4):045001, 18, 2008.

\bibitem{egloffe_these}
A.-C. Egloffe.
\newblock {\em {\'Etude de quelques probl\`emes inverses pour le syst\`eme de
  Stokes. Application aux poumons.}}
\newblock PhD thesis, Universit\'e Paris VI, 2012.

\bibitem{rr}
A.-C. Egloffe.
\newblock {Lipschitz stability estimate in the inverse Robin problem for the
  Stokes system}.
\newblock {\em Inria research report, RR-8222}, 2013.

\bibitem{fabre}
C.~Fabre and G.~Lebeau.
\newblock Prolongement unique des solutions de l'equation de {S}tokes.
\newblock {\em Comm. Partial Differential Equations}, 21(3-4):573--596, 1996.

\bibitem{quarteroni}
A.~Quarteroni and A.~Veneziani.
\newblock Analysis of a geometrical multiscale model based on the coupling of
  {ODE}s and {PDE}s for blood flow simulations.
\newblock {\em Multiscale Model. Simul.}, 1(2):173--195 (electronic), 2003.

\bibitem{sincich}
E.~Sincich.
\newblock Lipschitz stability for the inverse {R}obin problem.
\newblock {\em Inverse Problems}, 23(3):1311--1326, 2007.

\bibitem{stefanov}
P.~Stefanov and G.~Uhlmann.
\newblock Linearizing non-linear inverse problems and an application to inverse
  backscattering.
\newblock {\em J. Funct. Anal.}, 256(9):2842--2866, 2009.

\bibitem{irene}
I.~E. Vignon-Clementel, C.~A. Figueroa, K.~E. Jansen, and C.~A. Taylor.
\newblock Outflow boundary conditions for three-dimensional finite element
  modeling of blood flow and pressure in arteries.
\newblock {\em Comput. Methods Appl. Mech. Engrg.}, 195(29-32):3776--3796,
  2006.

\end{thebibliography}
\bibliographystyle{plain}
\end{document}